\documentclass[final]{amsart}

\usepackage{amssymb}
\usepackage{showkeys}
\usepackage{xypic}
\usepackage{enumerate}

\numberwithin{equation}{section}

\theoremstyle{remark}
\newtheorem{example}{Example}[section]

\newtheorem{definition}[example]{Definition}
\newtheorem{construction}[example]{Construction}
\newtheorem{remark}[example]{Remark}

\theoremstyle{plain}
\newtheorem{proposition}[example]{Proposition}

\newtheorem{corollary}[example]{Corollary}
\newtheorem{theorem}[example]{Theorem}
\newtheorem{lemma}[example]{Lemma}

\newcommand{\coh}[2]{{\rm H}^{#1}(#2)}

\newcommand{\isoarrow}{\,\tilde{\longrightarrow}\,}

\newcommand{\struct}[1]{{\mathcal O}_{#1}}

\newcommand{\spec}{{\rm Spec}\,}

\newcommand{\gm}{{\mathbb G}_m}

\newcommand{\bun}{{\rm Bun}}
\newcommand{\bD}{{\textbf D}} 
\newcommand{\bx}{{\textbf x}} 
\newcommand{\by}{{\textbf y}} 
\newcommand{\bz}{{\textbf z}} 
\newcommand{\lri}{{\longrightarrow}}
\newcommand{\soc}{{\rm Soc}}
\newcommand{\rk}{{\rm rk}}

\newcommand{\parmu}{\text{\rm par-}\mu}
\newcommand{\fk}{{\rm Fields}_k}
\newcommand{\sets}{{\rm Sets}}
\newcommand{\ed}{{\rm ed}}
\newcommand{\triv}{{\rm triv}}
\newcommand{\flag}{{\rm Flag}}
\newcommand{\supp}{{\rm supp}}

\newcommand{\qcoh}{{\rm QCoh}}

\newcommand{\sX}{{\mathfrak X}}
\newcommand{\sG}{{\mathfrak G}}
\newcommand{\sF}{{\mathfrak F}}
\newcommand{\m}{{\mathfrak m}}
\newcommand{\E}{{\mathcal E}}
\newcommand{\G}{{\mathcal G}}
\newcommand{\F}{{\mathcal F}}
\newcommand{\W}{{\mathcal W}}

\newcommand{\e}{{\rm cd}}

\newcommand{\RR}{{\mathbb R}}
\newcommand{\Ext}{{\rm Ext}}
\newcommand{\SL}{{\rm SL}}
\newcommand{\BG}{{\mathcal BG}}
\newcommand{\parb}{{\rm par}}

\newcommand{\eF}{{\textbf F}}

\begin{document}

\title[Essential Dimension of Stacks of Parabolic Bundles]{The Essential Dimension
of Stacks of Parabolic Vector Bundles over Curves}

\author[I. Biswas]{Indranil Biswas}

\address{School of Mathematics, Tata Institute of Fundamental
Research, Homi Bhabha Road, Bombay 400005, India}

\email{indranil@math.tifr.res.in}

\author[A. Dhillon]{Ajneet Dhillon}

\address{Department of Mathematics, University of Western Ontario,
London, Ontario N6A 5B7, Canada}

\email{adhill3@uwo.ca}

\author[N. Lemire]{Nicole Lemire}

\email{nlemire@uwo.ca}



\thanks{Second and third authors partially supported by NSERC.
The third author acknowledges the support of a
Swiss National Science Foundation International short
visit research grant.}

\begin{abstract}
We find upper bounds on the essential dimension of the moduli stack
of parabolic vector bundles over a curve. When there is no parabolic 
structure,
we improve the known upper bound on the essential dimension of the usual
moduli stack. Our calculations also give lower bounds on the essential dimension of the
semistable locus inside the moduli stack of vector bundles of rank $r$ and degree $d$
without parabolic structure.
\end{abstract}

\maketitle

\section{Introduction}

Essential dimension is a numerical invariant of an algebraic
or geometric object defined over a base field $k$ which roughly speaking
measures its
complexity in terms of the smallest number of parameters needed to define
the object over $k$. It was originally defined by Reichstein and Buhler, see 
\cite{buhler:97, reichstein:00}, in the context
of finite and then algebraic groups. The definition was then 
rephrased into functorial language and generalized by Merkurjev.
See~\cite{berhuy:03,merk:09} for surveys on this topic.

Denote by $\fk$ the category of field extensions of
$k$. Let $F\,:\,\fk\,\lri\,\sets$ be a functor. We say that
$a\,\in\, F(L)$ is \emph{defined over a field} $K\,\subseteq \,L$ if
there exists a $b\,\in\, F(K)$ so that $r(b)\,=\,a$ where $r$ is the
restriction
\[
F(K)\,\lri \, F(L)\, .
\]
The \emph{essential dimension} of $a$ is defined to be 
$$
\ed(a)\,\stackrel{\rm def}{=} \, {\rm Min}_K \,\text{\rm tr-deg}_k K\, ,
$$
where the minimum is taken over all fields of definition $K$ of
$a$.

The \emph{essential dimension} of $F$ is defined to be
$$
\ed(F) \,= \,{\rm sup}_a \ed(a)\, ,
$$
where the supremum is taken over all $a\,\in\, F(K)$ and $K$ varies
over all objects of $\fk$.

Reichstein and Buhler first considered the essential dimension
of the functor
$$F_G\,:\,\fk\,\longrightarrow\, \sets\, , ~\, K\,
\longmapsto\, H^1(K,G)\, ,$$
where $H^1(K,\, G)$ is the set of isomorphism classes of $G$-torsors
over ${\rm Spec}(K)$. Then
$\ed(G)\,:=\, \ed(F_G)$ is a numerical invariant of the algebraic group.
Upper and lower bounds on $\ed(G)$ have been a topic of much 
study since then. See~\cite{reichstein:10} for a survey of work on
essential dimension and the particular cases of the essential
dimension of finite and algebraic groups.

For an algebraic stack
$\sX\,\longrightarrow\, {{\rm Aff}}_k$
we obtain a functor
$$F_{\sX}\,:\, \fk\,\lri\, \sets $$
which sends any $K$ to the set of isomorphism classes
of objects in $\sX(K)$. 
\emph{ The essential dimension of } $\sX$ is defined to be the
essential dimension of this functor. We denote this
number by $\ed_{k}(\sX)$.
This definition was first made by Brosnan, Reichstein and Vistoli,
see \cite{brosnan:09}, in which they studied the essential dimension of an 
algebraic stack. Note that this is indeed a generalization of the 
original definition of essential dimension. In particular,
for the classifying stack of an 
algebraic group, $\BG$, we have $\ed_k(\BG)\,=\,\ed_k(G)$
since $F_{\BG}\,=\,H^1(-,G)$. Furthermore, viewed in this light, essential 
dimension can be considered as a generalization of the dimension of an algebraic 
variety. Indeed there are two ways to define the dimension of a variety, via the 
transcendence degree of its function
field or via Krull dimension. Essential dimension is a lift to stacks of the transcendence degree
definition.

Brosnan, Reichstein and Vistoli in ~\cite{brosnan:09} used stack-theoretic results to prove a genericity
theorem computing the essential dimension of a Deligne-Mumford stack $\sX$,
of finite type with finite inertia, from the dimension of its coarse moduli 
space $M$
and the essential dimension of the generic gerbe $\sX_K$, where $K$ is the 
function field of $M$ (see Theorem \ref{t:main} for more details).
They were then able to apply these results to obtain 
better bounds for the essential dimension of some algebraic and finite
groups.

We fix a base field $k$ of characteristic zero that may not be algebraically 
closed. We further fix
 a smooth projective geometrically connected curve $X$ of genus $g$ at least $2$ over this field. We will also assume
that $X$ has at least three $k$-rational points. This technical hypothesis is needed in the proof of 
Theorem \ref{t:functorpoly}. 

For an algebraic group $G$,
we denote by $\bun^G_X$ the moduli stack of $G$-torsors over the curve $X$.
We are interested in studying the 
essential dimension of $\bun^G_X$ for an algebraic group $G$.
Note then that the 
essential dimension of $\bun^G_X$ is the essential dimension of the 
functor sending a field extension $K/k$ to the set of isomorphism
classes of $G$-torsors over $X_K$.

We concentrate on the case $G\,=\,\SL_r$. In this case,
$\SL_r$ torsors over the curve $X$ are vector bundles over $X$ with fixed
isomorphism of the determinant bundle with the trivial bundle on $X$.
In~\cite{dl:09}, we obtained upper bounds on $\ed(\bun^{\SL_r}_X)$
and the essential dimensions of some related moduli stacks
such as $\bun^{r,d}_{X}$, the moduli stack of vector
bundles over $X$ of rank $r$ and degree $d$. Our
upper bounds on $\ed(\bun^{r,d}_X)$ 
depended on the genus $g$ and were quartic polynomials of the rank $r$.

In this paper, we 
 denote by $\bun^{r,d}_{X,\bD}$, the moduli stack of vector bundles of rank $r$
and degree $d$ on $X$ with parabolic structure
along some reduced divisor (see \S~\ref{s:pbundles}).
Note that in the special case of rational weights (see \S~\ref{s:pbundles} for 
definitions and more details), vector bundles on $X$ with parabolic structure along some reduced
divisor $D$ were shown by Biswas in ~\cite{biswas:97} to correspond to
 certain orbifold bundles over $X$ where $\bD$ is the 
ramification divisor.

 Our aim 
is to compute an upper bound on the essential
dimension of $\bun^{r,d}_{X,\bD}$. When the divisor is empty, the
moduli stack coincides with $\bun^{r,d}_X$,
the moduli stack of vector bundles of 
rank $r$ and degree $d$ over $X$. Dhillon and Lemire found bounds on the 
essential
dimension of $\bun^{r,d}_X$ in~\cite{dl:09}. Our
results improve the upper bound obtained there; this is explained in Remark 
\ref{r:comparison}. In particular, our new upper bound is now
a quadratic polynomial in the rank $r$, which depends on $g$.
Further, by carefully choosing our parabolic data we
are able to find lower bounds on the essential dimension of the semistable locus of the usual 
moduli stack (see \S~\ref{s:lower}).
There was no previously known lower bound better than the trivial bound
given by the dimension of the coarse moduli space (see Theorem
\ref{t:main}).

In order to find an upper bound on the essential dimension of $\bun^{r,d}_{X,\bD}$
we first show that $\bun^{r,d,s}_{X,\bD}$, namely the moduli stack of stable 
vector bundles
of rank $r$ and degree $d$ with parabolic structure along $\bD$,
is a Deligne Mumford stack satisfying the conditions of the 
genericity theorem mentioned above.
We are able to apply this theorem to $\bun^{r,d,s}_{X,\bD}$ 
 using the 
calculation in \cite{bd:09} of Brauer group of the moduli space and
some facts about twisted sheaves due to Lieblich
~\cite{lieblich:05}.

To pass from the essential dimension of $\bun^{r,d,s}_{X,\bD}$
to that of $\bun^{r,d,ss}_{X,\bD}$, the moduli stack of semistable vector bundles
of rank $r$ and degree $d$ with parabolic structure along $\bD$,
we use the socle filtration, a finite filtration of a semistable
parabolic bundle with polystable parabolic bundles as quotients (see \S~\ref{s:socle} for details).
This was the key to improving the bounds in \cite{dl:09} as there we
used the the Jordan-H\"older filtration, a finite filtration of 
a semistable parabolic bundle with stable parabolic bundles as quotients. Unlike the 
Jordan-H\"older filtration, the socle
filtration is Galois invariant, so it exists over the base field. This sidesteps the major difficulty in
\cite{dl:09}.

We then use the Harder-Narasimhan filtration, a finite filtration
of a parabolic bundle with semistable parabolic bundles as quotients,
to pass from the essential dimension of the semistable locus
to that of the full moduli stack, $\bun^{r,d}_X$.

The other main ingredient is the correspondence set up in \cite{biswas:97} between parabolic bundles
and orbifold bundles. This allows us to compute extensions of parabolic bundles in terms of orbifold
bundles. Finally we need the orbifold Riemann-Roch theorem, originally proved in \cite{toen:99}, to bound
the dimensions of these groups.

The key results that compute upper bounds are Theorem \ref{t:full}, 
Proposition \ref{p:ssbound}, Proposition \ref{p:growth} and 
Proposition \ref{p:same}. The first 
two results bound the essential dimension in terms of an auxiliary 
function. The last two tell us
about the growth rate of this function.
The lower bound is given in Theorem \ref{t:lower}.

In summary, we have the following (best known) bounds for the essential
dimension of the
moduli stack of vector bundles of rank $r$ and degree $d$ :

\begin{theorem}
Let $X$ be a smooth projective curve of genus at least 2 and with at least 3
$k$-points.
Suppose that $p$ is a prime such that $p^l$ divides $r$ and $d$. 
 We have
$$
 (r^2-1)(g-1) + p^l - 1 + g \,\le\,\ed(\bun^{r,d}_{X}) \le r^2g\, .
$$
\end{theorem}
This result follows by combining Theorem \ref{t:full}, Theorem \ref{t:lower}, 
Proposition \ref{p:same} and Proposition \ref{p:growth}.

Upon introducing parabolic structure we obtain the following result :

\begin{theorem} 
Suppose that $p$ is a prime such that $p^l$ divides $l(\bD)$.
 We have
$$
 (r^2-1)(g-1) + p^l  - 1 + g + 
\sum_{y\in\bD} \dim \flag_{y}(\bD) \, \le\, \ed(\bun^{r,d}_{X}) \le 
F_{g,\bD}(r)\, .
$$
\end{theorem}

The reader is referred to section 13 for the definition of the function
$F_{g,\bD}(r)$. The number $l(\bD)$ is defined in section 6.
Once again this results from combining Theorem \ref{t:full}, 
Theorem \ref{t:lower}, Proposition \ref{p:same} and Proposition \ref{p:growth}.

It is worth noting at this point that the main conjecture of \cite{colliot:07} is closely related to 
calculating the essential dimension of the moduli stack of vector bundles via 
Theorem \ref{t:main2}.

\section*{Acknowledgments}
The authors would like to thank the referee for very useful comments.

\section{Essential Dimension}

In this section, we will 
recall some theorems from
\cite{brosnan:09}, including the genericity theorem mentioned in the 
introduction, that will be needed in the future. We assume for the
remainder of this section that $\sX/k $ is a Deligne-Mumford stack,
of finite type,
with finite inertia. By \cite{keel:97}, such a stack has a coarse moduli 
space $M$. The first result that we shall need is the following theorem
proved in \cite{brosnan:09}.

\begin{theorem}\label{t:main}
Recall that our base field has characteristic zero.
Suppose $\sX$ is 
smooth and connected. Let $K$ be the field of rational
functions of $M$, and let $\sX_K = \spec(K) \times_{M} \sX$
be the base change. Then
$$
\ed_{k}(\sX) \,=\, \dim M + \ed_K(\sX_K)\, .
$$
\end{theorem}

\begin{proof}
See \cite[Theorem 6.1]{brosnan:09}.
\end{proof}

The stack $\sX_K/K$ is called the generic gerbe. In the case
where this gerbe is banded by $\mu_n$, more can be said about
$\ed_K(\sX_K)$. 

Let $\sG$ be a gerbe over a field $K$ banded
by $\mu_n$. There is an associated $\gm$-gerbe over $K$, denoted by $\sF$, coming from the canonical inclusion
$\mu_n\,\hookrightarrow\, \gm$.
It gives a torsion class
in the Brauer group ${\rm Br}(K)$. 
The index of this class is called the \emph{index} of the gerbe, and is denoted 
by ${\rm ind}_K(\sG)\, =\, d$. There is a Brauer-Severi variety $P/K$ of 
dimension $d-1$
whose class maps to the class of $\sG$ via the connecting homomorphism
$$
\coh{1}{\spec(K),{\rm PGL}_d}\,\lri\, \coh{2}{\spec(K),\gm}
$$
for the exact sequence $e\,\longrightarrow\, \gm \,\longrightarrow\,{\rm GL}_d
\,\longrightarrow \,{\rm PGL}_d\,\longrightarrow\, e$.

Let $V$ be a smooth and proper variety over $k$.
The set $V(k(V))$ is the collection of 
rational endomorphisms of $V$ defined over $k$.
Define
$$ \e_k(V) \,=\, \inf\{\dim \overline{{\rm im}(\phi)} \mid
\phi\in V(k(V)) \}\, .$$
The number $\e_k(V) $ is called the \emph{canonical dimension} of $V$.
We recall another theorem from \cite{brosnan:09}.

\begin{theorem}\label{t:main2}
In the above situation, if $d>1$, then
$$
\ed_K(\sG) \,=\, \e_K(P) + 1\, ,
$$
and
$$
\ed_K(\sF) \,=\, \e_K(P) \, .
$$
\end{theorem}

See \cite[Theorem 4.1]{brosnan:09} for a proof.

\begin{corollary} \label{c:prime}
In the above situation, if ${\rm ind}(P)=p^r$ is a prime power, we have
$$
\ed_K(\sG) \,=\, {\rm ind}_K(P)
$$
and 
$$
\ed_K(\sF) \,=\, {\rm ind}_K(P) - 1
$$
\end{corollary}

\begin{proof}
See \cite[Theorem 2.1]{kar:00} and \cite{merk:03}.
\end{proof}

\section{Parabolic Bundles}
\label{s:pbundles}

\begin{definition}
\label{d:ppoint}
 A \emph{parabolic point} $\bx$ on $X$ consists of a triple
$$
(x\, ,\{k^x_i \}_{i=1}^{n(x)}\, , \{\alpha^x_i\}_{i=1}^{n(x)})\, ,
$$
where $x$ is a $k$-point of $X$, the $k_i^x$ are positive integers called the
\emph{multiplicities} and the $\alpha^x_i$ are rational numbers, called the
\emph{parabolic weights} (or simply {\it weights}). The weights are required to 
satisfy the following 
condition :
$$
0\,\le\, \alpha^x_1 \,<\, \alpha^x_2\,<\,\cdots \,<\, \alpha^x_{n(x)} \,<
\, 1\, .
$$ 
\end{definition}

\begin{definition}
\label{d:datum}
 A \emph{parabolic datum} $\bD$ on $X$ consists of a finite collection of 
parabolic points $\bx_j\,=\, (x_j\, ,\{k^{x_j}_i \}_{i=1}^{n(x_j)}\, , 
\{\alpha^{x_j}_i\}_{i=1}^{n(x_j)})$, so
$$
\bD \,=\,\{\bx_1\, , \bx_2\, , \cdots\, , \bx_s \}\, ,
$$
such that $\sum_{i=1}^{n(x_j)} k^{x_j}_i$ is independent of $j$.
We require the points to be pairwise distinct, that is
$x_j\,\ne \,x_i$ for $j\,\ne\, i$.
\end{definition}

The \emph{support} of the datum is defined to be the reduced divisor $x_1 + 
\ldots + x_s$. We denote this divisor
by $|\bD|$.

\begin{definition}
\label{d:pbundle}
Fix a parabolic datum $\bD$ on $X$.
If $S$ is a scheme then a \emph{family of parabolic bundles} $\F_*$ on $X$ 
parameterized by $S$ with parabolic datum $\bD\,=\,\{\bx_1\, , \bx_2\, , 
\cdots\, , \bx_s 
\}$ consists of a vector bundle
$\F$ on $X\times S$ together with filtrations by vector subbundles
$$
\F|_{\{x_j\}\times S}\, =\, F^{x_j}_1(\F)\,\supset\, F^{x_j}_2(\F)\,\supset\, 
\cdots \,\supset\, F^{x_j}_{n(x_j)}(\F)
\,\supset\, F^{x_j}_{n(x_j)+1}(\F)\,=\,0
$$
with $F^{x_j}_{i}(\F)$ locally free of rank 
$$k^{x_j}_{i} + k^{x_j}_{i+1} + \cdots +k^{x_j}_{n(x_j)}\, .$$
\end{definition}

The weight $\alpha^{x}_i$ is associated with $F^{x}_i(\F)$.
This definition forces $\rk(\F)\,=\,\sum_{i=1}^{n(x)} k^{x}_i$ for each
$x\,\in \, \supp(|\bD|)$. 

When $S$ is reduced to a point we call $\F_*$ a parabolic bundle.

\begin{definition}
\label{d:psubbundle}
Suppose that $\F_*$ is a parabolic bundle with datum $\bD$. 
A \emph{parabolic subbundle} 
$\F'_*\,=\,(\F',\{F^{x}_i(\F')\, :\, i\,=\,1,\cdots,n'(x),x\,\in\, 
\supp(|\bD|)\})$ of $\F_*$ is a parabolic bundle with datum $\bD'$ such that
\begin{enumerate}
\item $|\bD|\,=\,|\bD'|$
\item $\F'$ is a subbundle of $\F$
\item for each point $x$ in the support, the weights $\{\alpha'^{x}_i\}_{i=1}^{n'(x)}$ are a subset of the weights $\{\alpha^{x}_i\}_{i=1}^{n(x)}$
\item if $m$ is maximal so that $F_i^{x}(\F')\,\subseteq\,F_m^{x}(\F)$, then
${\alpha'}^{x}_i \,=\, \alpha^{x}_m$.
\end{enumerate}
\end{definition}

Given a parabolic bundle $\F_*$ on $X$ with parabolic datum 
$$\bD \,=\, \{\bx_1\, , \bx_2\, , \cdots\, , \bx_s \}\, ,$$
we define the \emph{parabolic degree} of $\F_*$ to be the rational number
$$
\text{par-deg}(\F_*) \,=\, \deg(\F) + \sum_{j=1}^s\sum_{i=1}^{n(x_j)} 
k^{x_j}_{i}\alpha^{x_j}_{i}\, .
$$
The \emph{parabolic slope} is defined to be $\text{par-}\mu(\F_*) \,=\, 
\text{par-deg}(\F_*)/{\rm rk}(\F)$.

Denote by $\overline{k}$ an algebraic closure of the ground field $k$. 

We say the $\F_*$ is \emph{semistable} (respectively, \emph{stable}) if 
for every parabolic subbundle $\E_*$ of $(\F_{\overline{k}})_*$ with $0\,<\, {\rm 
rk}(\E_*) \,<\, {\rm rk}(\F_*)$, we have 
$$
\parmu(\E_*)\le \parmu((\F)_*) \quad (\text{respectively, } 
\parmu(\E_*) < \parmu((\F_{\overline{k}})_*)\, ,
$$
where $(\F_{\overline{k}})_*$ is the base change of $\F_*$.

The usual arguments show that an arbitrary parabolic bundle has a unique maximal 
destabilizing parabolic subbundle $\E_* \,\subseteq\, (\F_{\overline{k}})_*$ of 
maximal
parabolic slope. The uniqueness implies that all the Galois conjugates
$\sigma^*(\E_*)\, \subset\, (\F_{\overline{k}})_*$ coincide.
Hence this subbundle is defined over the ground field $k$ so that a base extension is not required in the
definition of semistable parabolic bundles.

\begin{construction}
\label{c:rfiltration}
Let $\F_*$ be a parabolic bundle with datum $\bD$. We wish to construct 
 a bundle $\F_t$ for each $t\,\in\, \RR$. 
Set $D \,=\, |\bD|$. For each $t\in \RR$ with $0\,\le\, t \,<\,1$, we construct a 
coherent sheaf $V_t(\F_*)$ supported on $D$ by letting the component of 
$V_t(\F_*)$ on $x_j$ be the subspace $F_i^{x_j}(\F)$ of 
$\F|_{x_j}$, where $\alpha^{x_j}_{i-1} \,<\, t\,\le\, \alpha^{x_j}_i$;
if $t\,>\,\alpha^{x_j}_{n(x_j)}$, then the component of $V_t(\F_*)$ on $x_j$ is 
defined to be zero.
Taking preimages of $V_t(\F)$ gives a sheaf $\F_t$ with
$\F\,\supseteq \,\F_t \,\supseteq \,\F(-D)$. We can extend this construction to 
$t\,\in\, \RR$ by defining $\F_t \,=\, \F_s(-\lfloor t \rfloor D)$, where 
$s\,=\, t-\lfloor t \rfloor$.

This collection $\{\F_t\}_{t\in \mathbb R}$ is decreasing; it has jumps at 
rational numbers only. Also, it is periodic, more precisely, $\F_t(-D)\,=\, 
\F_{t+1}$, and 
it is left continuous. It is clear that $\{\F_t\}_{t\in \mathbb R}$
uniquely determines the parabolic bundle.
\end{construction}

Following \cite{yokogawa:95} we will now define an abelian category $\qcoh(X,D,N)$. 

First, we consider $\RR$ as a category using its natural structure as an ordered set.
An $\RR$-filtered sheaf is a functor
$$
\RR^{\rm op}\rightarrow \qcoh(X).
$$
Morphisms of $\RR$-filtered sheaves are just natural transformations. 
Given $s\in\RR$ and an $\RR$-filtered sheaf $\F_*$ we can define a shifted sheaf by setting
$$
\F[s]_t = \F_{s+t}.
$$
Given a quasicoherent sheaf $\G$ and an $\RR$-filtered sheaf $\
F_*$, we define the $\RR$-filtered sheaf
$\G\otimes \F_*$ by taking the tensor product pointwise, ie
$$
(\G\otimes \F_*)_t = \G\otimes\F_t.
$$

We identify a full abelian subcategory $\qcoh(X,D,N)$  of the category of $\RR$-filtered sheaves.
The objects of this
category are $\RR$-filtered sheaves $\F_*$ such that
\begin{enumerate}[(i)]
 \item we have a natural isomorphism
$$
j:\F_*\otimes\struct{X}(-D) \isoarrow \F[1]_*.
$$                                    
such that the induced map
$$
\F_{s+1} \cong \F_s(-D) \rightarrow \F_s
$$
is the natural inclusion.
\item For $ l/N < s \le (l+1)/N$ we have that the natural maps
$
\F_{l+1/N}\rightarrow \F_{s}
$
are  isomorphisms.
\end{enumerate}

Fix a reduced effective divisor $D\,=\,\sum_{i=1}^s x_i$ on $X$, where each $x_i$ 
is a $k$-rational point.

Denote by ${\rm PVect}(X,D,N)$ the category of parabolic bundles with parabolic
datum only inside the support of $D$ and parabolic weights integer multiples of 
$\frac{1}{N}$.
The morphisms in this category are given by the following definition.

\begin{definition}
Suppose that $\F$ and $\F'$ are parabolic bundles with parabolic bundles with 
parabolic data $\bD$ and $\bD'$.
Suppose that $|\bD|\,=\,|\bD'|$.
A \emph{morphism of parabolic bundles} $f_*\,:\,\F_*\,\longrightarrow\, \F'_*$ is 
a morphism 
$f\,:\,\F\,\longrightarrow\,\F'$ of underlying bundles such that for every parabolic point 
$\bx$ we have
$$
f_x ( F^x_i(\F) ) \,\subseteq\, F^x_j(\F')
$$
whenever $\alpha_i^x \,>\, {\alpha'}^x_j$.
\end{definition}

We identify ${\rm PVect}(X,D,N)$ with a full subcategory of the abelian category
$\qcoh(X,D,N)$ via \ref{c:rfiltration}.

\section{Our Stacks}\label{s:stacks}

Let $\bD$ be a parabolic datum on $X$. We will denote by $\bun^{r,d}_{X,\bD}$ 
the moduli stack of
rank $r$ degree $d$ parabolic bundles with datum $\bD$. Note that the weights 
only play a role when defining
stability and semistability. Hence this stack is just a fibered product of flag 
varieties over the moduli stack $\bun^{r,d}_X$ of vector bundles of rank $r$ 
and degree $d$ without parabolic structure. 

Fix a line bundle $\xi$ on $X$. We denote by $\bun^{r,\xi}_{X,\bD}$ the moduli 
stack of parabolic bundles 
with fixed identification with $\xi$ of the top exterior power of the 
underlying vector bundle. Precisely, there is a Cartesian square
$$
\xymatrix{ 
\bun^{r,\xi}_{X,\bD} \ar[r] \ar[d] & \spec(k) \ar[d] \\
\bun^{r,d}_{X,\bD} \ar[r] & \bun^{1,d}_X.
}
$$
Here $\bun^{1,d}_X$ is the moduli stack of line bundles of degree $d$, the 
right vertical arrow corresponds to the line bundle $\xi$ and the bottom 
horizontal arrow is the determinant map.
As stability and semistability are open conditions, see \cite[page 
226-228]{mehta:80}, there are various open substacks, $\bun^{r,d,s}_{X,\bD}\, , 
\bun^{r,d,ss}_{X,\bD}\, , \bun^{r,\xi,s}_{X,\bD}$ and
$\bun^{r,\xi,ss}_{X,\bD}$.

We explain explicitly what $\bun_{X,\bD}^{r,\xi}$ is. The objects of the 
category fibered in groupoids over a scheme
$S$ consist of pairs $(\F,\phi)$, where $\F$ is
a family of parabolic bundles of rank $r$ on $X\times S$ and $\phi$ is an isomorphism
$$
\phi\,:\, \wedge^r\F\,\isoarrow\, \xi\, .
$$
The isomorphisms in the groupoid over $S$ are isomorphisms of parabolic bundles
compatible with the trivializations.

The stack $\bun_{X,\bD}^{r,\xi}$ is somewhat unnatural as $\xi$ is not a parabolic 
line bundle. However it is a natural stepping stone in understanding the 
essential 
dimension of the stack $\bun_{X,\bD}^{r,d}$. Below we will see that 
$\bun_{X,\bD}^{r,\xi,s}$ is a smooth Deligne-Mumford stack with finite inertia so 
that Theorem \ref{t:main} applies. We will be able to compute the period and index 
of its generic gerbe and apply Theorem \ref{t:main2} to understand its essential 
dimension. 

\begin{proposition}
 The stack $\bun_{X,\bD}^{r,\xi}$ is smooth.
\end{proposition}

\begin{proof}
First recall that the moduli stack $\bun_{X}^{r,\xi}$ of vector bundles is smooth.
If $\bD \,=\,\{\bx_1\, ,\cdots\, , \bx_s\}$, then set 
$$
\bD' \,=\, \{\bx_1\, ,\cdots\, , \bx_{s-1}\}\, .
$$
The morphism forgetting one parabolic point 
$$
\bun_{X,\bD}^{r,\xi} \,\lri \,\bun_{X,\bD'}^{r,\xi}
$$
is representable with fibers flag varieties. Hence the above morphism is 
smooth and the result follows by induction.
\end{proof}

\begin{remark}
\label{r:quotient}
The stack $\bun_{X,\bD}^{r,d,s}$ is in fact a global quotient stack. For 
simplicity, in this remark we will assume that
$\bD\,=\,\{\bx\}$ with multiplicities $k_1\, ,\cdots\, , k_n$. 

{}From \cite[page 226]{mehta:80},
the family of stable parabolic bundles of rank $r$ and degree $d$ is a bounded 
family. We may find an integer $N$ so that for every 
$n\ge N$, we have
\begin{itemize}
\item ${\rm H}^1(X,\F(n))\,=\,0$, and

\item $\F(n)$ is generated by 
global sections for every vector bundle
underlying a stable parabolic bundle.
\end{itemize}
Let $Q$ be the corresponding quot scheme. Let $\W$ be the universal bundle
on $Q\times X$. There is a flag variety $F$ over $Q$ parametrizing flags of 
$\W_x$ of type 
$k_1,\cdots, k_n$. To give an $S$-point of $F$ is the same as giving a 
quotient :
$$
\pi^*_X{\mathcal O}_{X}(-N)^m \,\twoheadrightarrow\, \F
$$
on $S\times X$ and a flag of $\F|_{S\times\{x\}}$. There is a Zariski open 
subset $\Omega$ parametrizing quotients
that are stable as parabolic bundles. We have
$$
[\Omega/{\rm GL}_m] \,=\, \bun_{X,\bD}^{r,d,s}\, .
$$

The stack $\bun_{X,\bD}^{r,\xi,s}$ is also a global quotient stack. There is a 
$\gm$-torsor $I$ over the
determinant $\xi$ locus in $\Omega$ parametrizing isomorphisms of the 
determinant with $\xi$ (see \cite[page 29]{laumon:00}). Then
$$
[I/{\rm GL}_m] \,=\,\bun_{X,\bD}^{r,\xi,s}\, .
$$
\end{remark}

\begin{proposition}
\label{p:basic}
 Let $\F$ be a family of stable parabolic bundles on 
$\spec(R)\times X$. Then all parabolic endomorphisms of $\F$ are
scalar multiplication by elements of $R$.
\end{proposition}

\begin{proof}
This is well known when $R$ is a field. We will explain how to deduce this result
from the case of a field.

There exists a natural inclusion
$$
\epsilon \,:\, R \,\hookrightarrow \,\coh{0}{X_R, \text{End}(\F)}
$$
that we wish to show to be an isomorphism. By flat base change, we may assume 
$R\,=\,(R,\mathfrak{m})$ is
local. Via Nakayama's Lemma we need to show that
$$
\overline{\epsilon}\,:\, R/\mathfrak{m}\,\hookrightarrow\,\coh{0}{X_R, 
\text{End}(\F)}\otimes_R R/\mathfrak{m}
$$
is surjective. But by the field case, the composition 
$$
R/\mathfrak{m}\,\longrightarrow\, \coh{0}{X, \text{End}(\F)}\otimes_R 
R/\mathfrak{m} \,\longrightarrow\,
\coh{0}{X_{R/\mathfrak{m}}, \text{End}(\F_{R/\mathfrak{m}})}
$$
is surjective. The result follows from the base change theorem,
\cite[Ch. III, Theorem 12.11]{hartshorne:77}.
\end{proof}

\begin{theorem}
 The stack $\bun_{X,\bD}^{r,\xi,s}$ is a smooth Deligne-Mumford stack with
finite inertia.
\end{theorem}

\begin{proof}
This follows by applying Remark \ref{r:quotient} and Proposition \ref{p:basic}.
\end{proof}

It follows that $\bun_{X,\bD}^{r,\xi,s}$ has a coarse moduli space that
we shall denote by $M(X,\bD,r,\xi)^s$. 

For each parabolic point $\bx$ in the datum $\bD$,
we denote by $\flag_\bx(\bD)$ or $$\flag(k^{x}_1,k^{x}_2,\cdots,k^{x}_{n(x)})$$
the flag variety determined by the multiplicities of $\bx$.
Explicitly, if the multiplicities at $\bx$ are $k^x_1,k^x_2,\cdots 
,k^x_{n(x)}$, then 
$\flag_\bx(\bD)$ is the flag variety parametrizing flags
$$
V_1\,\supseteq\, V_2\,\supseteq \,\cdots \,\supseteq \,V_{n(x)}
$$
of a fixed vector space $V_1$
with $\dim V_1 \,=\, \sum_{i=1}^{n(x)} k^x_i$ and $\dim V_i/V_{i+1}
\,=\, k^x_i$. 

Employing this notation, we compute the dimension of the moduli space to be
\begin{equation} \label{e:dimension}
\dim M(X,\bD,r,\xi)^s \,= \, \dim \bun^{r,\xi}_{X,\bD}
\,= \, (r^2 - 1)(g-1) + \sum_{x\in\bD} \dim {\rm Flag}_x(\bD)\, .
\end{equation}
By Proposition \ref{p:basic}, this makes the stack into a gerbe banded
by the $r$-th roots of unity over the moduli space.

\section{Period and Index}

For a parabolic datum $\bD\,=\,(D, \{k_i^x,\alpha_i^x\,:\,x\,\in\, 
\supp(D)\, ,i\,=\, 1\, ,\cdots\, ,n(x)\}) $ we define
an integer 
$$
l(\bD)\,=\, \gcd (\deg(\xi)\, ,r\, , \{k^{x}_{i}\,:\,x\,\in\, 
\supp(D)\, ,i\,=\,1\, ,\cdots\, ,n(x)\})\, 
,
$$
where the greatest common divisor is taken over the rank degree and all 
multiplicities of all parabolic weights at all parabolic parabolic points.

\begin{theorem}
\label{t:bd}
Assume the base field is the complex numbers.
 The period of the gerbe
$$
\bun_{X,\bD}^{r,\xi,s} \,\lri\, M(X,\bD,r,\xi)^s
$$
is $l(\bD)$.
\end{theorem}

\begin{proof}
This follows from \cite{bd:09}. Note that the gerbe of splittings
of the Severi-Brauer variety in \cite{bd:09} is easily identified with
the moduli stack.
\end{proof}

We will see below that the above result is true over any field of 
characteristic zero.

A useful tool for understanding the difference
between the period and the index is the notion of a
twisted sheaf. A \emph{twisted sheaf} on a $\gm$-gerbe
$\sG\,\longrightarrow\, X$ is a coherent sheaf $\F$ on $\sG$
such that inertial action of $\gm$ on $\F$ coincides with
natural module action of $\gm$ on $\F$. We spell out the meaning of
this statement
in the next paragraph.

 Suppose that we have a $T$-point $T\longrightarrow X$ and an object
$a$ of $\sG$ above this point. Part of the data of the coherent
sheaf $\F$ is a sheaf $\F_a$ on $T$. These sheaves are required to
satisfy compatibility conditions on pullbacks for morphisms in the category
$\sG$. In particular, every object $a$ of the gerbe $\sG$ has an action
of $\gm$ and hence there is an action of $\gm$ on $\F$. The above 
definition says that the action of $\gm$ on $\F$ should be the same as the 
$\gm$-action coming from the fact that $\F$ is an $\struct{\sG}$-module.

\begin{example} \label{e:twisted}
We have a $\mu_r$-gerbe 
$$
\bun^{r,\xi,s}_{X,\bD} \,\lri\, M(X,\bD,r,\xi)^s\, .
$$
It gives rise to a $\gm$-gerbe via the natural group homomorphism
$\mu_r\,\lri\,\gm$. We denote this gerbe by $\bun^{r,\xi,s,\gm}_{X,\bD}$. One can 
describe this stack explicitly. The objects over a scheme $S$ are families $\F_*$ 
of stable parabolic bundles on $S\times X$ such that $\wedge^r \F$ is isomorphic to
$$
\pi_S^*{\mathcal L} \otimes \pi_X^*\xi
$$
for some line bundle $\mathcal L$ on $S$.
There is a universal stable parabolic bundle $\W_*$ on 
$$
\bun^{r,\xi,s,\gm}_{X,\bD} \times X\, .
$$
The data that makes up $\W_*$ consists of a vector bundle $\W$ of rank $r$ on 
$$
\bun^{r,\xi,s,\gm}_{X,\bD} \times X\, .
$$
and, for each parabolic point $\bx$ in the datum $\bD$,
a universal flag 
$$
F_{1}^x(\W) \,\supseteq \,\cdots \,\supseteq\, F_{n(x)}^x(\W)
$$
on 
$$
\bun^{r,\xi,s,\gm}_{X,\bD} \times \{x\}.
$$

As the automorphism group of a stable parabolic bundle is multiplication by a 
scalar, we see that each of
these bundles produces a twisted sheaf. Fix a $k$-rational point $y\in X$. We have on
$$
\bun^{r,\xi,s,\gm}_{X,\bD}
$$
the following twisted sheaves
\begin{enumerate}
\item $\W_{\bun^{r,\xi,s,\gm}_{X,\bD}}\times \{y\}$ of rank $r$,

\item $F_{i}^x(\W)$ of rank $k^x_{i}+ \ldots + k^x_{n(x)}$ for each $\bx
\,\in\, \bD$, $1\,\le\, i\,\le\, n(x)$.
\end{enumerate}

There is one more twisted sheaf that will be of importance to us below.
We have a projection
$$
p\,:\,\bun^{r,\xi,s,\gm}_{X,\bD} \times X\,\lri\, \bun^{r,\xi,s,\gm}_{X,\bD}\, .
$$
The sheaf
$$
p_*\W\otimes \struct{X}(Ny)
$$
is locally free for large $N$ and gives us another twisted sheaf on the stack.
Using Riemann-Roch, it has rank
$$
\deg(\xi) + Nr + r(1-g)\, .
$$
\end{example}

We will need the following :

\begin{proposition} \label{p:twisted}
Let $\sG\,\longrightarrow\,\spec(K)$ be a $\gm$-gerbe over a field.
Then the index of $\sG$ divides $m$ if and only if there is 
a locally free twisted sheaf on $\sG$ of rank $m$.
\end{proposition}

\begin{proof}
See \cite[Proposition 3.1.2.1]{lieblich:05}.
\end{proof}

\begin{theorem}
\label{t:index}
 The index equals the period for the gerbe
$$
\bun_{X,\bD}^{r,\xi,s} \,\lri\, M(X,\bD,r,\xi)^s\, .
$$
\end{theorem}

\begin{proof}
 Over the complex numbers the result follows from Theorem \ref{t:bd}, Example
\ref{e:twisted} and Proposition \ref{p:twisted}. 
For a non algebraically closed field
we proceed as follows. Set $e\,=\,\gcd(\deg(\xi)\, ,r\, ,\{k^{x}_{j}\,:\, x\,
\in \, \supp(|\bD|)\, ,j\,=\,1\, ,\cdots\, ,n(x))\}$. From
Example \ref{e:twisted} it
follows that the index divides $e$. After base change to an algebraically closed field
we find that $e$ is the period, using a Lefschetz principle. This means that 
the period of the original gerbe was larger than $e$. However
the period always divides the index.
\end{proof}

\section{The Essential Dimension of the Stable Locus}

Let $g$ be the genus of $X$.

\begin{theorem}
\label{t:stable}
Set 
\begin{equation}\label{ld}
l(\bD)\,:=\,\gcd(\deg\xi\, , r\, ,k^{x}_{i})\, ,
\end{equation}
where the gcd ranges over all possible multiplicities of all parabolic weights
at all the parabolic points in the datum $\bD$.
If $l(\bD)\,>\, 1$ then the essential dimension of $ \bun_{X,\bD}^{r,\xi,s}$ is 
bounded above by
$$
l(\bD) + \dim M(X,\bD,r,\xi)^s \,=\, l(\bD) + (r^2-1)(g-1)+\sum_{\bx\in 
\bD}\dim \flag_{\bx}(\bD)\, .
$$
This upper bound is an equality when $l(\bD)=p^l>1$ is a prime power.

For any $l(\bD)$ we have
$$
\ed(\bun_{X,\bD}^{r,\xi,s,\gm}) \,\le\, l(\bD)-1 + (r^2-1)(g-1)+\sum_{\bx\in 
\bD}\dim \flag_{\bx}(\bD)\, .
$$
This upper bound is an equality when $l(\bD)$ is a prime power.
\end{theorem}

\begin{proof}
Recall that the dimension of the moduli space was computed in equation
\eqref{e:dimension} on page 9.

Also recall that for a field extension $L/k$, an $L$-point of $\bun_{X,\bD}^{r,\xi,s,\gm}$ is the same as a 
family of stable parabolic bundles
with prescribed data and determinant $\xi$ while an $L$-point of
$\bun_{X,\bD}^{r,\xi,s}$ is such a similar family and fixed 
trivialization of the top exterior power.

If $l(\bD)\,>\,1$, then the upper bounds follow from Theorem \ref{t:main2}, 
Theorem \ref{t:main} and Theorem \ref{t:index}. To deduce the result for 
$l(\bD)\,=\,1$, observe that the gerbe 
 $$
 \bun_{X,\bD}^{r,\xi,s,\gm} \,\lri\, M(X,\bD,r,\xi)^s
 $$
is neutral as the period is one. Hence there is a universal bundle on the 
moduli space and the theorem follows.

Suppose that $l(\bD)$ is a prime power. We show that 
$$
\ed(\bun_{X,\bD}^{r,\xi,s,\gm}) = l(\bD)-1 + (r^2-1)(g-1)+\sum_{\bx\in 
\bD}\dim \flag_{\bx}(\bD)\, .
$$
The assertion for  $\bun_{X,\bD}^{r,\xi,s}$ is similar.
Denote by $K$ the function field of the moduli space.
In the case that $l(\bD)$ is a prime power, applying  \ref{c:prime} and Theorem \ref{t:index},
 we can find a family $\F$ defined over a field
$L$ containing $K$ so that 
$$
\text{\rm tr-deg}_K L = l(\bD)-1,
$$
and this family is not compressible over $K$. If it descends to an extension $L'/k$
then $L'$ must contain $K$ by properties of the moduli space. It follows that 
$$
\text{\rm tr-deg}_K L = \text{\rm tr-deg}_K L'
$$
 Hence the result in
the prime power case. 
\end{proof}

\begin{corollary}
 \label{c:stable}
\begin{enumerate}
\item[(i)] The essential dimension of $ \bun_{X,\bD}^{r,d,s}$ (degree d, and not
fixed determinant) is bounded above by 
$$
l(\bD)-1 + \dim M(X,\bD,r,d)^s \,= \,l(\bD)-1 + (r^2-1)(g-1)+\sum_{\bx\in 
\bD}\dim \flag_{\bx}(\bD)+g\, .
$$
\item[(ii)] When $l(\bD)\,=\,p^l$ a prime power the above inequality is an equality.
\end{enumerate}
\end{corollary}

\begin{proof}
If $l(\bD)=1$ the moduli space has a universal bundle and we are done, so we may 
assume $l(\bD)>1$. 

For (i), 
 suppose that we have a family $\F$, of stable parabolic vector bundles of rank $r$ and
degree $d$ over a field $K$. By adjoining at most $g$
parameters to $k$,
the base field, we may assume that $\det \F$ is defined over our new base field $k'$
and then apply the above result. In other words, the family compresses to a field $L$ with
\begin{eqnarray*}
\text{\rm tr-deg}_{k'} L &\le &   l(\bD) - 1 + (r^2-1)(g-1)+\sum_{\bx\in 
\bD}\dim \flag_{\bx}(\bD)\\
\end{eqnarray*}

For (ii), denote by $k(J^d(X))$ the function field of the degree $d$ picard variety.
We have on $X\otimes_k k(J^d(X))$ a line bundle $\xi$ obtained by pulling back the 
Poincar\'e bundle. We apply the above theorem to the stack $\bun_{X,\bD}^{r,\xi,s,\gm}$
to obtain a field extension $K/ k(J^d(X))$ and a family $\F$ over this extension that
cannot be compressed to a field extenion $L/k(J^d(X))$ of smaller transcendence degree.
If this family compresses to a family $\F'$ over $L'/k$ then by considering the map to $J^d(X)$
induced by $\det\F'$ we see
that $L'$ contains $k(J^d(X))$. Hence the result.

\end{proof}

\section{Some Linear Algebra}
\label{s:linalg}
We will use the results from this section to find an upper bound on
the essential dimension of the polystable locus. The results of this section
allow us to pass from polystable to stable by adding appropriate parabolic
structure.

Let $V$ be a finite dimensional vector space of dimension $r$. We equip $V$ with 
two full flags 
$$V\,=\,F_1^x\,\supset\, F_2^x\,\supset\, \cdots \,\supset\,F_r^x\,\supset\,
F_{r+1}^x\,=\,0$$
and
$$V\,=\,F_1^y\,\supset \,F_2^y\,\supset \,\cdots\, \supset\,F_r^y\,\supset\, 
F_{r+1}^y\, =\, 0$$
with $\dim F_i^x\, =\, \dim F_i^y \,=\, r-i+1$. We say that the flags are 
\emph{generic} if 
$$
\dim(F_i^x+F_j^y)\,=\,\min(r,\dim(F_i^x)+\dim(F_j^y))\, .
$$
It is clear that generic flags exist. Fix a $1$-dimensional subspace 
$l\subseteq V$ with 
$$
\dim(l+F_i^x + F_j^y)\,=\, \min(1+\dim(F_i^x)+\dim(F_j^y),r)\, .
$$
We say that the triple $(F_*\, , F_*'\, ,l)$ is \emph{generic} if the flags are 
generic and the subspace $l$ satisfies the above condition. It is easy to see 
that generic triples exist and any generic pair of flags can be completed to a 
generic triple.

For a subspace $W\subseteq V$ define the \emph{degree} of $W$ to be
\begin{eqnarray*}
\deg_V(W) &=& \sum_{i=1}^r ((i-1)(\dim(W\cap F^x_i)-\dim(W\cap F^x_{i+1})) \\
 &&    + \sum_{i=1}^r ((i-1)(\dim(W\cap F^y_i)-\dim(W\cap F^y_{i+1})) \\& = & \sum_{i=2}^{r} \dim(W\cap F_i^x) + \sum_{i=2}^{r} \dim(W\cap F^y_i).
\end{eqnarray*}

We also need the notation
$$\deg (W)\, =\, \deg_V (W) + (r-1)\dim (l \cap W)\, .$$
Note that $\deg_V (W)$ only depends on the first two flags and not the
line.

Let us set $d_i^x(W)\,=\,\dim(W\cap F^x_i)$ and $d_i^y(W)\,=\,\dim(W\cap F^y_i)$.

We define the \emph{slope} of $W$ to be $\mu(W)\,=\,\deg (W)/\dim (W)$.

\begin{lemma}\label{l:degbound}
For a generic pair of flags $\{F^x_i\}$ and $\{F^y_i\}$, We have 
$$
(W\cap F^x_{i})\oplus (W\cap F^y_{r+2-i})\,\subseteq\, W
$$
so that $d^x_i(W)+d^y_{r+2-i}(W)\,\le\, \dim(W)$ for all $2\,\le\, i\,\le\, r$.

In the case $l\,\subseteq\, W$, we have 
$$
l\oplus (W\cap F^x_{i})\oplus (W\cap F^y_{r+3-i})\,\subseteq\, W
$$
so that $d^x_i(W)+d^y_{r+3-i}(W)\,\le\, \dim(W)-1$ for all $2\,\le\,
i\,\le\, r+1$. In particular, $d^x_2(W)\,\le\, \dim(W)-1$ and $d^y_2(W)\,
\le \, \dim(W)-1$.
\end{lemma}

\begin{proof}
This follows from the definition of a generic pair of flags.
\end{proof}

\begin{proposition}
\label{p:vstable}
We have
$$
\mu(V)\,>\, \mu(W)
$$
for every proper subspace $W$.
\end{proposition}

\begin{proof}
Note that $\deg(V)\,=\,2\sum_{i=1}^{r-1}i+(r-1)\,=\,r^2-1$ so that 
$\mu(V)\,=\,r-\frac{1}{r}$.

First consider the case that $W\cap l\,=\,\{0\}$. Then by Lemma~\ref{l:degbound},
we have
\begin{eqnarray*}
\deg(W) &=& \deg_V W \\
		&=& \sum_{i=2}^{r} (d^x_i(W) + d^y_{r+2-i}(W)) \\
		&\le & (r-1) \dim W.
\end{eqnarray*}
So $\mu(W)\,\le\, r-1<\mu(V)$.

Then consider the case in which $l\subseteq W$.
By Lemma~\ref{l:degbound}, we have
\begin{eqnarray*}
\deg(W) &=& \deg_V W +(r-1) \\
		&=& (\sum_{i=3}^r (d_i^x(W) + d^y_{r+3-i}(W))) +
(d^x_{2}(W)+d^y_{2}(W)) +(r-1) \\
		&\le & r(\dim(W)-1)+(r-1).
\end{eqnarray*}
So $\mu(W)\,\le \,r-\frac{1}{\dim(W)}\,<\,\mu(V)$.
\end{proof}

\section{The Socle}\label{s:socle}

\begin{definition}
Let $K$ a field containing $k$.
 We say that a parabolic bundle $\F$ on $X_K$ is \emph{polystable} if 
$\F\otimes_K \overline{K}$ is a direct sum of stable parabolic bundles
of the same parabolic slope.
\end{definition}

\begin{proposition}
Let $\F_*$ be a semistable parabolic bundle on $X_{\overline{K}}$ with 
parabolic slope $\mu$. Then there exists a unique maximal polystable subbundle 
with parabolic slope $\mu$. We call this bundle the socle of $\F_*$ and write 
$\soc(\F_*)$. If $\F_*$ is defined over $K$ then so is $\soc(\F_*)$.
\end{proposition}

\begin{proof}
Over algebraically closed fields for bundles without
parabolic structure the proof can be found in \cite{huybrechts:97}.
To add parabolic structure one can use the parabolic orbifold correspondence, see
below Theorem \ref{p:equivalence}. Note that if a semistable vector bundle $\F$ 
has
a group action the uniqueness of the socle implies that the socle is preserved by that
action.

To see the assertion about ground fields notice that $\soc(\F_*)$ will always be 
defined over some finite Galois extension $L/K$ with Galois group $G$. The Galois 
action will descend to the socle, as observed above.
\end{proof}

Consider the functor 
$$
\eF \,=\, \eF^{{\rm poly}\, , r \, ,d }_{X,\bD} \,:\, {\rm Fields}_k 
\,\longrightarrow \, \sets
$$
with 
$$
\eF(K) \,=\, \{ \text{families of polystable bundles on }X_K\text{ of degree 
}d\text{ and rank } r\}/\sim\, ,
$$
where $\sim$ is the equivalence relation given by isomorphism of families.
We will need to say something about the essential dimension of this functor.
The main idea is that we can turn a polystable vector bundle into a stable one
by adding parabolic structure at three points.

Let $\F\,\in\,\eF(K)$. Choose three $k$-points $x,y,z\in X$ that are not 
parabolic points. (Recall that we assumed that they exist in the introduction.)
Choose a trivialization of $\F$ in a Zariski neighborhood of the three points.
Using the trivialization we can identify the fiber over the three points with
a common vector space $V$. Then $V$ is a $r$ dimensional vector space.
We turn these points into parabolic points by defining full flags at $x$ and $y$ 
to be $$V\,=\,F_1^x\,\supset \,F_2^x\,\supset\, \cdots F_r^x\,\supset\, 
F_{r+1}^x\,=\,\{0\}$$
and
$$V\,=\,F_1^y\,\supset\, F_2^y\,\supset\,\cdots\,F_r^y\, \supset\, 
F_{r+1}^y\,=\,\{0\}$$
so that $k_i^x\,=\,k_i^y\,=\,1$ for all $i\,=\,1\, ,\cdots\, ,r$.
Choose the flag at $z$ to be $$V\,=\,F_1^z\,\supset\, F_2^z\,=\,l\,\supset\,
F_3^z\,=\, 0\, ,$$
where $l$ is a line in $V$ so that $k_1^z\,=\,r-1$ and $k_2^z\,=\,1$.
The weights for $x$ and $y$ are chosen to be
$\alpha_i^x\,=\,(i-1)\epsilon$, $i\,=\, 1,\cdots,r$, and the weights for $z$ are 
chosen to be $\alpha_1^z\,=\,0$ and $\alpha_2^z\,=\,(r-1)\epsilon$,
where $\epsilon$ is so chosen that the largest weight is smaller than $1$.
The corresponding parabolic points are denoted $\bx,\by$ and $\bz$.

Let
$$
E \,=\, \flag_\bx{\F} \times_K \flag_\by(\F)\times_K \flag_\bz(\F)\, .
$$
On $E\times X$ there is a universal extension of the quasiparabolic structure of 
$\F$ to the three new
points. This means that there is a parabolic bundle $\E_*$ on $E\times X$ with 
datum $\bD'\,=\,\bD\cup\{\bx\, ,\by\, ,\bz\}$. 
Note by construction, the parabolic slope of the 
new parabolic bundle $\E_*$ is 
$$\text{par-}\mu(\E_*)\,=\, \text{par-}\mu(\F_*)+
\sum_{i=1}^rk_i^x\alpha_i^x+\sum_{i=1}^rk_i^y\alpha_i^y+ 
\sum_{i=1}^2k_i^z\alpha_i^z\,=\,\mu+\mu(V)\epsilon$$
where $\mu(V)$ is defined in Section~\ref{s:linalg} .
Any parabolic subbundle $\E'_*$ of $\E_*$ has parabolic slope
$$\text{par-}\mu(\E'_*)\,=\,\mu +\mu(W)\epsilon\, ,$$
where $W$ is the common fiber of the subbundle $\E_*'$ and hence
is a vector subspace of $V$. Then by Proposition~\ref{p:vstable}, we have that 
$\E_*$ is a stable parabolic bundle.
There is an open subscheme $E^s\subseteq E$ where this bundle is stable. 

\begin{lemma}
The open subscheme $E^s$ is not empty.
\end{lemma}

\begin{proof}
It suffices to show that $E^s\otimes_K\overline{K}$ is not empty so we may 
assume 
that
$$
\F \,= \,\F_1\oplus \cdots \oplus \F_l
$$
with the $\F_i$ non isomorphic stable bundles of the same slope $\mu$. We may 
find an open $U$ of $X_K$ which contains $x,y$ and $z$ such that
$\F\vert_U$ is trivial. Then we apply the argument above to obtain the result.
\end{proof}

\begin{theorem}
 \label{t:functorpoly}
We have
\begin{eqnarray*}
\ed(\eF) &\le & (r^2-1)(g-1)+\sum_{\bx\in \bD}\dim \flag_{\bx}(\bD)+g + r^2-1\\
&=&r^2g+ \sum_{\bx\in \bD}\dim \flag_{\bx}(\bD)\, .
\end{eqnarray*}
\end{theorem}

\begin{proof}
As above, we may add parabolic structure at $x,y$ and $z$ to obtain a stable 
parabolic bundle. Since $l(\bD')\,=\,1$ (defined in \eqref{ld} and 
$$\dim \flag_{\bx}(\bD')\,=\,\dim \flag_{\by}(\bD')\,=\,r(r-1)/2\, ,~\, 
\dim \flag_{\bz}(\bD')\,=\,r-1$$
we may apply Corollary~\ref{c:stable} to obtain the result.
\end{proof}

\section{Orbifold Bundles and Orbifold Riemann-Roch}

Let $Y$ be a smooth projective curve with an action of the finite group $\Gamma$ 
defined over $k$.
If $\E$ is a $\Gamma$ bundle on $Y$ the cohomology groups $\coh{*}{Y,\E}$ are 
naturally representations of the group $\Gamma$. We define $\chi(\Gamma,Y,\E)$ to 
be the equivariant Euler
characteristic. Precisely, it is the class 
$$
\chi(\Gamma\, ,Y\, ,\E) \,=\, [\coh{0}{Y,\,\E}] - [\coh{1}{Y,\,\E}]
$$
in the $K$-ring of representations of $\Gamma$. The
orbifold Riemann--Roch theorem is a formula for this class.

\begin{theorem}
\label{t:rr} Suppose that $k=\overline{k}$.
Consider the projection
$$
\pi\,:\,Y\,\lri\, X\,=\,Y/\Gamma\, .
$$
For each $y\in Y$ write $e_y$ for the ramification index of $\pi$ at $y$ and 
$\Gamma_y$ for the isotropy group at $y$. We have a character
$$
\chi_y \,:\, \Gamma_y \,\longrightarrow \,\gm
$$
coming from the action of $\Gamma_y$ on the cotangent space $\m_y/\m_y^2$.

We have
$$
\begin{array}{ccr}
|\Gamma|\chi(\Gamma,Y,\E) & =& (|\Gamma|(1-g)\rk(\E) + \deg(\E))[k[\Gamma]] - \\
&&\sum_{y\in Y}\sum_{d=0}^{{e_y-1}} d[{\rm Ind}^\Gamma_{\Gamma_y}(\E|_y \otimes \chi_y^d)].
\end{array}
$$
\end{theorem}

\begin{proof}
See \cite{kock:05}. This formula can also be deduced from \cite[Theorem 
4.11]{toen:99} by considering the
morphism of quotient stacks
$$
[Y/\Gamma] \,\longrightarrow \, B\Gamma\, .
$$
\end{proof}

We now recall the main result of \cite{biswas:97} in the case of curves.
Consider a reduced divisor $D\,=\,\sum_{i=1}^s x_i$ on $X$, where $x_i$ are 
$k$-rational points.

Recall that the positive integer $N$ was chosen in Section 3 so that the 
parabolic weights were integer multiples of $1/N$.
Consider a curve  $Y$ with an action of a finite group $\Gamma$ such that 
$Y/\Gamma \,=\,X$. There is a projection $\pi\,:\,Y\,\lri\, X$. We further assume 
that for each $x\,\in\, {\rm supp}(D)$, we have
$\pi^*(x) \,=\, kN(\pi^*x)_{\rm red}$ for some positive integer $k$.
Denote by ${\rm Vect}_\Gamma^D(Y,N)$
the full subcategory of $\Gamma$-bundles on $Y$ with $W\in {\rm Vect}_\Gamma^D(Y)$
if and only if 
\begin{itemize}
\item
for all geometric points $y$ in $Y$ with $y\,\in\, {\rm supp}((\pi^*D)_{\rm 
red})$, and for each $\gamma\,\in\, \Gamma_y$ in the isotropy
group, $\gamma^N$ acts on the  fiber $W_y$  trivially, and

\item
for all geometric points $y$ in $Y$ with $y\notin {\rm supp}((\pi^*D)_{\rm 
red})$, the action of the isotropy group $\Gamma_y$ 
on the fiber $W_y$ is trivial.
\end{itemize}

\textbf{Note : } We have not asserted that such a $Y$ exists over our base field $k$.
If such a curve $Y$ with $\Gamma$ action exists then we will say that it 
\emph{splits} the parabolic structure on $X$.

Recall that ${\rm PVect}(X,D,N)$  is the category of parabolic bundles with parabolic
datum only inside the support of D and parabolic weights integer multiples of $1/N$.
The category ${\rm PVect}(X,D,N)$ is a tensor category. To define the tensor 
product, it is convenient to think
of parabolic bundles as being an appropriate family $\{ \F_t\}_{t\in \RR}$, as 
described in Construction \ref{c:rfiltration}. Then $(\F\otimes\F')_t$
is the subsheaf of $i_*i^*\F\otimes\F'$ generated by
$$
\F_a \otimes \F'_b\, , \quad a+b\ge t\, .
$$
Here we denote the inclusion $X\setminus D\hookrightarrow X$ by $i$. One checks 
that the resulting collection
$(\F_* \otimes \F'_*)_{t\in \RR}$ gives a bundle with parabolic datum $\bD$.

With these definitions ${\rm PVect}(X,D,N)$ becomes a  tensor category.
The unit $U$ for the tensor product is the parabolic bundle with 
$U_0\,=\,\struct{X}$ and $U_t\,=\,\struct{X}(-D)$ for $0<t<1$. It is readily
checked that we have an associative, commutative tensor structure with unit. 

\begin{theorem}
\label{p:equivalence}
There is a $k$-linear additive equivalence of tensor categories between ${\rm 
PVect}(X,D,N)$ and ${\rm Vect}_\Gamma^D(Y,N)$. 
\end{theorem}

\begin{proof}
This can be found in \cite[page 344]{balaji:01}, \cite{biswas:97} and \cite{borne:07}. Also see 
below for a description of one of the functors in this equivalence.
\end{proof}

We will denote the $\Gamma$-bundle associated to a parabolic bundle $\F_*$ by $\F_*^Y$.

There is a usual notion of exact sequence in the category ${\rm Vect}_\Gamma^D(Y,N)$.
There is also a notion of exact sequence in the category ${\rm PVect}(X,D,N)$ inherited from
the category $\qcoh(X,D,N)$.

\begin{proposition}
The equivalence in Theorem \ref{p:equivalence} preserves exact sequences.
\end{proposition}

\begin{proof}
We use the notation set up before Theorem \ref{p:equivalence}. 
Write $D\,=\,\sum_{i=1}^s x_i$ and $y_i\,= \,\pi^*(x_i)_{\rm red}$.
Set $\pi^*(x_i) \,=\, n_i y_i$.

It will be convenient to think of parabolic bundles in terms of Construction 
\ref{c:rfiltration}. From \cite{biswas:97}, the functor $M\,:\,
{\rm Vect}_\Gamma^D(Y,N)\,\lri\,{\rm PVect}(X,D,N)$ is given by the formula 
$M(W)\,=\,\E_t$, where
$$
\E_t \,=\, (\pi_*(W\otimes \struct{Y}(\sum_i^s\lceil -tn_i \rceil 
y_i))^\Gamma\, .
$$
The functor is clearly additive. It suffices to remark that an equivalence of 
abelian categories by additive functors must preserve exact sequences.
\end{proof}

\section{A Universal Construction}
 
Let $\E_*$ and $\E'_*$ be parabolic bundles with parabolic datum
$\bD$ and $\bD'$. We choose an integer $N$ so that all the weights
of both of the bundles are integer multiples of $\frac{1}{N}$. We denote by 
${\rm Ext}_{\rm par}(\E'_*, \, \E_*)$ 
the Yoneda Ext group in the category $\qcoh(X,D,N)$, where $D$ is chosen 
large enough to contain all parabolic points. Note that all such extensions will create
lie inside the category ${\rm PVect}(X,D,N)$ as $E_*$ and $E'_*$ have underlying vector bundles.
 It is a $k$-vector space and we 
view it as a variety. We would like to construct a universal extension on it. (A 
quick check shows that exact sequences are preserved by Baer sum and scalar 
multiplication.)

After some finite base extension $L/k$, there exists a group $\Gamma$, a 
smooth curve $Y$, and an action of  $\Gamma$ on $Y$ defined over $L$ such that 
$Y/\Gamma \,=\, X$ and $Y$ splits the
parabolic structures of $\E_*$ and $\E'_*$. Therefore, there is an equivalence 
of categories  between ${\rm PVect}(X,D,N)$ and
${\rm Vect}_\Gamma^D(Y,N)$. By further extension of $L$ we may assume that all 
representations of $\Gamma$ are defined over $L$.

\begin{proposition}
Let $\F$ and $\G$ be $\Gamma$ bundles on $Y$. There exists an $L$-vector space
$ {\rm Ext}^1_\Gamma(\F,\,\G)$, which we view as an $L$-variety, and an extension 
of $\Gamma$ sheaves on $ {\rm Ext}^1_\Gamma(\F,\,\G)\times Y$,
$$
0\,\lri\, \pi^*\G \,\lri\, \E \,\lri\, \pi^*\F \,\lri\, 0 \qquad (E)
$$
with the following universal property:
Given a scheme $f:V \lri \spec L$ and an extension
$$ 
0\,\lri\, f^*\G \,\lri \,\E' \,\lri\, f^*\F\,\lri\, 0 \qquad (E')
$$
of $\Gamma$ sheaves on $V\times Y$, there exists a unique $L$-morphism
$t\,:\,V\,\lri\, {\rm Ext}^1_\Gamma(\F,\G)$ with $t^*(E)\,\cong\, (E')$.
\end{proposition}

\begin{proof}
There exists a universal extension on ${\rm Ext}^1(\F,\G)$. This follows via
base change for cohomology. To obtain a universal $\Gamma$-extension, just restrict
this extension to 
$$
{\rm Ext}^1_\Gamma(\F,\G) \,\stackrel{{\rm defn}}{=}\, {\rm Ext}^1(\F,\G)_{\rm 
triv}\, .
$$
This proves the proposition.
\end{proof}

\begin{proposition}
\label{p:construction}
There exists a universal extension of parabolic bundles on 
$$
{\rm Ext}_{\rm par}(\E'_*,\, \E_*)\, .
$$
\end{proposition}

\begin{proof}
 We may assume that $L/k$ is Galois with group $G$. Using the equivalence in Theorem
\ref{p:equivalence} we see that there is a universal extension on the base extension
${\rm Ext}_{\rm par}(\E'_*, \,\E_*)\otimes L$. However, the universal extension
inherits a Galois action, in view of its universal property, and hence descends 
to $k$.
\end{proof}

We need to bound the dimension of ${\rm Ext}_{\rm par}(\E'_*,\, \E_*)$. The 
following lemma will be useful.

\begin{lemma}
\label{l:isotropy}
Let $\Gamma$ be a finite group and $\Gamma_y$ a cyclic subgroup of it with 
generator $\gamma$. Let $V$ be a finite dimensional representation of $\Gamma_y$ 
on which $\gamma^N$ acts trivially and $T$ a one dimensional representation
of $\Gamma$ whose restriction to  $\Gamma_y$ is faithful. Then
$$
\sum_{d=0}^{|\Gamma_y|-1} d\cdot \dim({\rm Ind}_{\Gamma_y}^\Gamma V \otimes 
T^d)_{\rm triv}
$$
is bounded by 
$$
(\dim V) |\Gamma_y|(1-\frac{1}{N})
$$
(here $W_{\rm triv}\,=\,W^{\Gamma_y} $ is the fixed part).
\end{lemma}

\begin{proof}
By base change we may assume that the ground field contains all roots of unity.
Let $\zeta$ be a primitive $|\Gamma_y|$-th root of unity.
Write $V \,=\, \bigoplus_{i=0}^{N-1} V_{(\frac{i}{N}|\Gamma_y|)}$, where the 
generator $\gamma$ acts as scalar 
multiplication by $\zeta^j$ on $V_j$.  Note that only these weight spaces
can occur as $\gamma^N$ acts trivially on $V$.
Then
$$
\sum_{d=0}^{|\Gamma_y|-1} d\cdot \dim({\rm Ind}_{\Gamma_y}^\Gamma V \otimes 
T^d)_{\rm triv}
$$
$$=\sum_{d=0}^{|\Gamma_y|-1}\sum_{i=0}^{N-1} d\cdot \dim(( 
V_{\frac{i}{N}|\Gamma_y|} \otimes T^d)_{\rm triv})$$
by Frobenius reciprocity.
But $\gamma$ acts as multiplication by the scalar $\zeta^{s+d}$ on $V_s\otimes 
T^d$.
We see that for fixed $i$, the set of invariants
$ (V_{\frac{i}{N}|\Gamma_y|} \otimes T^d)_{\rm triv}$ is non-zero  
if and only if either $(d,i)\,=\,(0,0)$ or $i>0$ with 
$d\,=\,|\Gamma_y|(1-\frac{i}{N})$. When $(V_{\frac{i}{N}|\Gamma_y|} \otimes 
T^d)_{\rm triv}$ is non-zero, $\dim((V_{\frac{i}{N}|\Gamma_y|} \otimes T^d)_{\rm 
triv})\,=\,\dim(V_{\frac{i}{N}|\Gamma_y|})$.
So the above sum
becomes
$$\sum_{i=1}^{N-1} |\Gamma_y|(1-\frac{i}{N}) \dim( V_{\frac{i}{N}|\Gamma_y|})
\,\le\, |\Gamma_y|(1-\frac{1}{N})\dim(V)$$
since 
$V = \bigoplus_{i=0}^{N-1} V_{\frac{i}{N}|\Gamma_y|}$.
\end{proof}

Recall from \cite{biswas:97} that a $\Gamma$-bundle is semistable if and only if the
underlying bundle is semistable. This fact follows from the uniqueness of the Harder-Narasimhan filtration.
We also need the following lemma.

\begin{lemma}\label{le-t-r}
 Let $\E$ be a semistable $\Gamma$-bundle on $Y$. Then 
$$
\dim \coh{0}{Y,\E}_{\rm triv} \,\le\, \left\{ 
\begin{array}{cc}
       0 &                            \text{if }\deg(\E)<0 \\
 \rk(\E) + \frac{\deg(\E)}{|\Gamma|} & \text{otherwise} 
\end{array}
\right.
$$
\end{lemma}

\begin{proof}
 The assertion is obvious when the degree is negative. We induct on the degree.
By extending $L$, we may find a point $y\in Y$ for which the isotropy 
subgroup, for the action of $\Gamma$, is trivial. 
We let $D$
be the divisor ${\rm orb}(y)$. The result follows from the exact sequence
$$
0\,\lri\, \E(-D) \,\lri\, \E \,\lri\, \E_D \,\lri\, 0\, .
$$
Note that the twist $\E(-D)$ is indeed semistable.
\end{proof}

\begin{proposition}
\label{p:bound}
Let $\F_*$ and $\G_*$ be parabolic bundles such that $\F_*^\vee\otimes \G_*$ is 
parabolic semistable and $\text{\rm par-deg}(\F_*^\vee\otimes \G_*)\, \geq\, 
0$. 
Then
$$
\dim {\rm Ext}_{\rm par}(\F_*, \,\G_*) \,\le\, 
\rk(\F_*)\rk(\G_*)g + (\deg|\bD|)(1-\frac{1}{N})\rk(\F_*)\rk(\G_*)\, ,
$$ 
where $g$ is the genus of $X$. (Recall the notation $|\bD|$ from Section 
\ref{s:pbundles}.)
\end{proposition}

\begin{proof}
We may pass to a field extension $L/k$ so that there is a $\Gamma$ cover $Y\,\lri 
\,X$ as in Theorem \ref{p:equivalence}. Write $\W_*^Y$ for the $\Gamma$-bundle 
associated to a parabolic bundle $\W_*$ under this equivalence of categories.
We need to compute the dimension of 
${\rm Ext}^1_\Gamma(\F_*^Y,\,\G_*^Y)$.  Note that the fact that
$$
\dim_L {\rm Ext}^1_\Gamma(\F_*^Y,\,\G_*^Y)\,=\,\dim_{\overline{L}} {\rm 
Ext}^1_\Gamma(\F_*^Y\otimes\overline{L},\,\G_*^Y\otimes\overline{L})
$$
allows us to pass to an algebraic closure and apply Theorem \ref{t:rr}
to the $\Gamma$ bundle $(\F^Y_*)^\vee\otimes\G^Y_*$. Then 
\begin{eqnarray*}\dim {\rm Ext}_{\rm par}(\F_*, \G_*)&=&h^1((\F^Y_*)^\vee\otimes\G^Y_*)_{\rm triv}\\
&=&h^0((\F^Y_*)^\vee\otimes\G^Y_*)_{\rm triv}-\chi(\Gamma,Y,(\F^Y_*)^\vee\otimes\G^Y_*)_{\rm triv}\\
&\le& 
\rk((\F^Y_*)^\vee\otimes\G^Y_*))g+ \\
 & &\frac{1}{|\Gamma|}\sum_{y\in Y} \sum_{d=0}^{e_y-1}d\dim[{\rm Ind}^\Gamma_{\Gamma_{y}}((\F^Y_*)^{\vee}\otimes G^Y_*)|_{y} \otimes \chi_{y}^d)_{\triv}]
\end{eqnarray*}
Here we applied Lemma \ref{le-t-r} under the
hypothesis that the $\Gamma$-bundle $(\F^Y_*)^\vee\otimes\G^Y_*$
is semistable of non-negative degree. This is true since, by hypothesis, the
corresponding parabolic bundle $\F^\vee\otimes \G$ is parabolic
semistable with non-negative parabolic degree.
We also applied Theorem \ref{t:rr} with the observation that the trivial
part $k[\Gamma]^\Gamma$ of the regular representation $k[\Gamma]$ has dimension 
one. Since $\rk((\F^Y_*)^\vee\otimes\G^Y_*))\,=\,\rk(\F)\rk(\G)$,
we need only bound the second term.
If $y\,\not\in \,\supp((\pi^*(D))_{\rm red})$, then Lemma \ref{l:isotropy} shows 
that the sum corresponding to $y$ vanishes as the isotropy group acts trivially
on $(\F^{\vee}\otimes \G)\vert_y$. If $N$ is an integer so that all weights are 
integer multiples of $\frac{1}{N}$ and $y\,\in\,\supp((\pi^*(D))_{\rm red})$,
then Lemma \ref{l:isotropy} shows that the sum corresponding to $y$ is
bounded by $(1-\frac{1}{N})(\rk(\F)\rk(\G))$ since in this case
the kernel of the homomorphism $\Gamma_y \longrightarrow {\rm 
GL}((\F^{\vee}\otimes \G)\vert_y)$
has order dividing $N$. As
$$\frac{1}{|\Gamma|}|\{y\in Y\,\mid\, y\in \supp((\pi^*(D))_{\rm red})\}|\,
\le \,\deg(|\bD|)\, ,$$
the proof is complete.
\end{proof}

Let $\bD$ be a parabolic datum on $X$. We denote by $N(\bD)$ the smallest 
positive  integer so that
the weights in $\bD$ are scalar multiples of $\frac{1}{N(\bD)}$. Set 
\begin{equation}\label{md}
(1-\frac{1}{N(\bD)})\deg(|\bD|) \,=\, M(\bD)\, .
\end{equation}

\begin{corollary}  
\label{c:bound}
Let $\E_*$ be a non-stable parabolic bundle of rank $r$
with parabolic data $\bD$. Let
$$0\subset (\E_1)_*\subset (\E_2)_*\subset\cdots\subset (\E_m)_*\,=\,\E_*$$
be the Harder-Narasimhan filtration of $\E_*$. Define
$(\E')_* \,:=\, (\E_{m-1})_*$. Then
$$\dim(\Ext^1_{\parb}((\E/\E')_*,\,(\E')_*))\,\le\, r'(r-r')(g+M(\bD))\, .$$
\end{corollary}

\begin{proof}
The Harder-Narasimhan filtration of $(\E'\otimes (\E/\E')^{\vee})_*$
as a parabolic bundle is
$$
\begin{array}{c}
0\subset (\E_1)_*\otimes (\E/\E_{m-1})_*^{\vee}\subset (\E_2)_*\otimes 
(\E/\E_{m-1})_*^{\vee}\subset\cdots\subset 
\\(\E_{m-1})_*\otimes (\E/\E_{m-1})_*^{\vee}=\E_*'\otimes 
(\E/\E_{m-1})_*^{\vee}\, .
\end{array}
$$
Note that 
\begin{eqnarray*}\text{par-}\mu(((\E/\E_{m-1})_*^{\vee}\otimes 
(\E_i/\E_{i-1})_*)
&=&\frac{(-r_md_i+d_mr_i)}{r_ir_m}\\
&=&\text{par-}\mu( (\E_i/\E_{i-1})_*))-\text{par-}\mu((\E_m/\E_{m-1})_*) \\
&>&0\, ,
\end{eqnarray*}
where $d_i$ is the parabolic degree of  $(\E_i/\E_{i-1})_*)$
and $r_i$ is its rank. So the previous proposition applies to 
each $((\E/\E_{m-1})_*^{\vee}\otimes (\E_i/\E_{i-1})_*)$
since this is a semistable parabolic bundle with positive parabolic slope.
We find that 
$$\dim(\Ext^1_{\text{par}}((\E/\E_{m-1})_*,\,(\E_i/\E_{i-1})_*))
\,\le\, (r_i)(r-r')(g+M(\bD))\, .$$

The result follows by a simple induction.
\end{proof}

\section{The Semistable Locus}

We define a function $F_{g,\bD}\,:\,{\mathbb N}\,\longrightarrow\,{\mathbb
N}$ recursively by
$$
F_{g,\bD}(r)\,=\,\max_{\substack{s+t=r\\ s\ge 0,\ t>0\\s,t\text{ 
integers}}}{F_{g,\bD}(s) + t^2g+ st(g+M(\bD))} 
$$
with $F_{g,\bD}(1)\,=\,g$ and $F_{g,\bD}(0)\,=\,0$, where
$M(\bD)$ is defined in \eqref{md}.

Let us record the following : 
\begin{lemma} \label{l:flagbound}
Consider positive integers $k_i$ with partitions
$s_i+t_i\,=\,k_i$. Here $s_i$ and $t_i$ are nonnegative. Then
$$
\dim \flag (s_1,\cdots, s_n) + \dim \flag (t_1, \cdots, t_n) 
\,\le\, \dim \flag (k_1, \cdots, k_n)\, .
$$
\end{lemma}

\begin{proof}
Recall that 
$$
\dim \flag (k_1, \cdots, k_n)\,=\,\sum_{i=1}^{n-1} k_i(k_{i+1} + \cdots +
k_n)\, .
$$
The result follows from this.
\end{proof}

\begin{proposition}\label{p:ssbound}
 We have 
$$
\ed(\bun^{r,ss,d}_{X,\bD})\,\le\, F_{g,\bD}(r) + \sum_{\bx\in \bD} 
\dim(\flag_\bx(\bD))\, .
$$
\end{proposition}

\begin{proof}
 We proceed by induction on the rank. The case of rank one is trivial.
Consider a parabolic bundle $\E_*$. If $\E_*$ is stable, then we are done
by Corollary~\ref{c:stable}.  
Suppose next that
$$
\soc(\E_*) \,=\, \E_*\, .
$$
Then by Theorem~\ref{t:functorpoly}, the bundle is defined over a field of 
transcendence degree at most
$$
r^2g + \sum_{\bx\in \bD}\dim(\flag_\bx(\bD))\, .
$$

In the remaining case there is an exact sequence
$$
0\,\longrightarrow\, \soc(\E_*)\,\lri\, \E_* \,\lri\, \F_*\,\lri\, 0\, .
$$
Suppose that $\bD_1$ and $\bD_2$ are the parabolic data for $\soc(\E_*)$ and 
$\F_*$ respectively. Lemma \ref{l:flagbound} shows that for every parabolic 
point $\bx$  we have
$$
\dim \flag_\bx(\bD_1) + \dim \flag_\bx(\bD_2) \,\le\, \dim \flag_\bx(\bD)\, .
$$
Let the ranks of $\soc(\E_*)$ and $\F_*$ be $t$ and $s$ respectively.
By Theorem \ref{t:functorpoly}, we know that $\soc(\E_*)$ is defined over a field 
of transcendence degree at most $t^2g +  \sum_{\bx\in \bD_1}\dim 
(\flag_{\bx}(\bD_1))$. By induction the parabolic bundle $\F_*$ is defined over a 
field of transcendence degree $F_g(s) + \sum_{\bx\in \bD_2}\dim 
(\flag_{\bx}(\bD_2))$. Let the compositum of these two fields be $K$.
We have
$$
\dim {\rm Ext}_{\rm par}(\F_*,\,\soc(\E_*))_{\rm triv}\,\le\, st(g + M(\bD))
$$ 
by Proposition \ref{p:bound}. Note that this result applies as 
$$
\F_*^\vee\otimes \soc(\E_*)
$$
is semistable of parabolic degree zero.
To obtain the result we apply Proposition \ref{p:construction}.
\end{proof}

\section{The Full Moduli Stack}

We form a function 
$$
G_{g,\bD}(r)\,=\,\max_{\substack{s+t=r\\ s\ge 0,\ t>0\\s,t\text{ integers}}}
F_{g,\bD}(t) + G_{g,\bD}(s) + st(g+ M(\bD))
$$
with $G_{g, \bD}(1)\,=\,g$ and $G_{g,\bD}(0)\,=\,0$.

\begin{theorem}
\label{t:full}
We have a bound
$$\ed(\bun^{r,d}_{X,\bD})\le G_{g,\bD}(r) + 
\sum_{\bx\in \bD}\dim(\flag_\bx(\bD))\, .
$$
\end{theorem}

\noindent
\textbf{Note : }The left-hand side in the inequality does not depend upon the
weights in the parabolic datum. Hence the inequality is true for all possible 
choices of weights. 

\begin{proof}
We prove this theorem by using induction on the rank $r$. The case of rank one
is trivial. Consider a parabolic bundle $\E_*$ and the exact sequence
$$
0\,\longrightarrow\,(\E_1)_*\,\lri\, (\E)_* \,\lri\,(\E_2)_*\,\lri\, 0\, .
$$
where $(\E_1)_*$ is the (destabilizing) parabolic proper subbundle of maximal
rank in the Harder-Narasimhan filtration of $\E_*$.
Suppose that $\bD_1$ and $\bD_2$ are the parabolic structures for $(\E_1)_*$ and 
$(\E_2)_*$ respectively. Then Lemma \ref{l:flagbound} shows that for every 
parabolic point $\bx$ we have
$$
\dim (\flag_\bx(\bD_1)) + \dim (\flag_\bx(\bD_2))
\,\le\, \dim (\flag_\bx(\bD))\, .
$$
Let the ranks of $(\E_1)_*$ and $(\E_2)_*$ be $t$ and $s$ respectively.
By Proposition \ref{p:ssbound} we know that $(\E_2)_*$ is defined over a field of 
transcendence degree at most $F_{g,\bD_1}(s) + \sum\dim \flag_\bx(\bD_1)$.
By induction, the parabolic bundle $(\E_1)_*$ is defined over a field of 
transcendence degree $G_{g,\bD_2}(t) + \sum_{\bx\in \bD_2}\dim 
\flag_{\bx}(\bD_2)$. Let the compositum of these two fields be $K$.  Note that 
$\deg(|\bD_i|)\le \deg(|\bD|)$ and $N(\bD_i)\,\le\, N(\bD)$ for $i\,=\,1,2$
so that $M(\bD_i)\,\le\, M(\bD)$ for $i\,=\,1,2$, and hence 
$F_{g,\bD_i}\,\le\, F_{g,\bD}$. 
Then $$\text{\rm tr-deg}K\,\le\, F_{g,\bD}(s)+G_{g,\bD}(t)+\sum_{\bx\in \bD}\dim 
(\flag_\bx(\bD))\, .$$
We have
$$
\dim {\rm Ext}_{\rm par}((\E_2)_*,\, (\E_1)_*)_{\rm triv} \,\le\, st(g+M(\bD))
$$ 
by Corollary \ref{c:bound}.
Let $W\,=\,{\rm Ext}((\E_2)_*, \,(\E_1)_*)_{\rm triv}$.  
The parabolic bundle $\E_*$ is defined over the function field $K'$
of a subvariety of this linear variety $W$. Then 
$${\rm trdeg}K'\,\le\, F_{g,\bD}(s)+G_{g,\bD}(t)+\sum_{\bx\in \bD}\dim 
(\flag_{\bx}(\bD))+st(g+M(\bD))\, .$$
The result follows.
\end{proof}

\section{Some Facts about $F_{g,\bD}$ and $G_{g,\bD}$}

Let
$$H_{g,\bD,r}(t)\,=\,F_{g,\bD}(t)+(r-t)^2g+(g+M(\bD))t(r-t)$$
so that
$$F_{g,\bD}(r)\,=\,\max_{0\le t\le r-1}H_{g,\bD,r}(t)\, .$$

\begin{proposition}
\label{p:growth}
If $g\,\le\, M(\bD)$, then for all $r\,\ge\, 0$, we have 
$$F_{g,\bD}(r)\,=\,\frac{r(r+1)}{2}g+\frac{r(r-1)}{2}M(\bD)\, .$$
If $g\,\ge\, M(\bD)$, then for all $r\,\ge\, 0$,
we have
$$F_{g,\bD}(r)\,=\,r^2g\, .$$
\end{proposition}

\begin{proof}

\noindent\textbf{Case 1:} $g\le M(\bD)$.

For $r\,=\, 0,1$, this follows from definition of $F_{g,\bD}$.
Assume the result for $0\,\le\, t\,<\,r$ by induction. 
Then by the inductive hypothesis, for all $0\,\le\, t\,\le\, r-1$, we have 
$$
H_{g,\bD,r}(t)\,= \,(\frac{t(t+1)}{2})g+(\frac{t(t-1)}{2})M(\bD) +
 (r-t)^2g+(g+M(\bD))t(r-t)\, .
$$

Simplifying this, we find that for all $0\,\le\, t\,\le\, r-1$, we have
$$
H_{g,\bD,r}(t)\,=\ 
\left(\frac{g-M(\bD)}{2}\right)t^2-(r-\frac{1}{2})(g-M(\bD))t+r^2g\, .
$$

Note that 
$$H_{g,\bD,r}(r-1)\,=\,\frac{r(r+1)}{2}g+\frac{r(r-1)}{2}M(\bD)\, .$$

So it suffices to prove the claim that  
$$\max_{0\le t\le r-1}H_{g,\bD,r}(t)\,=\, H_{g,\bD,r}(r-1)\, .$$

If $g\,=\,M(\bD)$, then it is clear that
$H_{g,\bD,r}(t)\,=\,r^2g$ for all $0\,\le\, t\,\le\, r-1$
so the claim  holds in this case.

Assume then that $g\,<\,M(\bD)$.

Consider the parabola that agrees with $H_{g,\bD,r}(t)$:
$$f(t)\,=\,\left(\frac{g-M(\bD)}{2}\right)t^2-(r-\frac{1}{2})(g-M(\bD))t
+r^2g\, .$$
Since
$$f'(t)\,=\,(g-M(\bD))(t-(r-1/2))\, ,$$
we have $f'(t)\,\ge\, 0$
if and only if $t\,\le\, (r-1/2)$ under the hypothesis $g-M(\bD)\,<\,0$. 
So in particular, $f(t)$ is increasing on the interval $0\,\le\, t\,\le\, r-1$.
But then since $H_{g,\bD,r}(t)\,=\,f(t)$,
we have
$$F_{g,\bD}(r)\,=\,\max_{0\le t\le r-1}H_{g,\bD,r}(t)\,=\,
H_{g,\bD,r}(r-1)\, .$$
as required.

\noindent\textbf{Case 2:} $g\,\ge\, M(\bD)$.

We will prove by induction on $r$ that 
$$F_{g,\bD}(r)\,=\,H_{g,\bD,r}(0)\,=\,r^2g$$
if $g\,\ge\, M(\bD)$.
The statement is true for $r\,=\,0,1$ by definition and 
since we have more generally that
$$H_{g,\bD,r}(0)-H_{g,\bD,r}(1)\,=\,(r-1)(g-M(\bD))\,\ge\, 0\, ,$$
 this shows that, in particular, we have
$$F_{g,\bD}(2)\,=\,H_{g,\bD,2}(0)\,=\,4g\, .$$

By the inductive hypothesis, we may assume that for $0\,\le\, t\,\le\, r-1$,
we have
$$H_{g,\bD,r}(t)\,=\,t^2g+(r-t)^2g+(g+M(\bD))t(r-t)\, .$$

Simplifying this, we obtain
$$H_{g,\bD,r}(t)\,=\,(g-M(\bD))(t(t-r))+r^2g$$
by the above. Since $(g-M(\bD))\,\ge\, 0$, we have $(g-M(\bD))(t(t-r))\,\le\, 0$ 
if $0\,\le\, t\,\le\, r-1$.
So $H_{g,\bD,r}(t)\,\le\, r^2g\,=\,H_{g,\bD,r}(0)$ if $0\,\le\,t\,\le\, r-1$.
This implies that $F_{g,\bD}(r)\,=\,r^2g$.

Observe that 
$$F_{g,\bD}(r)\,=\,r^2g\,=\,\frac{r(r+1)}{2}g+\frac{r(r-1)}{2}M(\bD)$$
if $g\,=\,M(\bD)$ so that the answers agree on the overlap.
\end{proof}

Recall that
$$
G_{g,\bD}(r)\,=\, \max_{\substack{s+t=r\\ s\ge 0,\ t>0\\s,t\text{ integers}}}
F_{g,\bD}(t) + G_{g,\bD}(s) + st(g+ M(\bD))
$$
and $G_{g,\bD}(1)\,=\,g,\, G_{g,\bD}(0)\,=\,0$.

\begin{proposition}
\label{p:same}
$F_{g,\bD}(r)\,=\,G_{g,\bD}(r)$ for all $r\,\ge\, 0$.
\end{proposition}

\begin{proof} 
The result is true by definition for $r\,=\,0,1$.
It suffices to prove that for all $0\,<\,s\,\le\, t$, we have 
$$F_{g,\bD}(s+t)-F_{g,\bD}(s)-F_{g,\bD}(t)-(g+M(\bD))st\,\ge\, 0$$
using the previous proposition.

\noindent\textbf{Case 1:} $g\,\le\, M(\bD)$.

Since 
$$F_{g,\bD}(r)\,=\,\frac{r(r+1)}{2}g+\frac{r(r-1)}{2}M(\bD)\, ,$$
we find that
$$F_{g,\bD}(s+t)-F_{g,\bD}(s)-F_{g,\bD}(t)-(g+M(\bD))st
\, =\,st(g+M(\bD))-st(g+M(\bD))=0\, .$$

\noindent\textbf{Case 2:}
$g\,>\, M(\bD)$.

Since 
$$F_{g,\bD}(r)\,=\,r^2g\, ,$$
we find that
$$F_{g,\bD}(s+t)-F_{g,\bD}(s)-F_{g,\bD}(t)-(g+M(\bD))st$$
$$=\, 2stg-st(g+M(\bD))=st(g-M(\bD))\,\ge\, 0\, .$$
This proves the proposition.
\end{proof}

\begin{remark}
\label{r:comparison}
The main result of \cite{dl:09} shows that 
$$
\ed(\bun_X^{r,d}) \,\le\, \lfloor h_g(r) \rfloor +  g\, .
$$
The function $h_{g}(r)$ is defined recursively by $h_g(1)\,=\,1$ and 
$$
h_g(r) - h_g(r-1)\,= \,r^3 - r^2 + \frac{r^4}{4}(g-1) + \frac{r}{2} + 
\frac{r^2g^2}{4} + \frac{1}{4}\, .
$$
(Note : solving the recursion would produce a quartic.)
Putting together Theorem \ref{t:full}, Proposition \ref{p:growth} and 
Proposition \ref{p:same} 
for $\bD\,=\,\emptyset$ and the original hypothesis $g\ge 2$ on the curve, we 
have
$$
\ed(\bun_X^{r,d}) \,\le\, r^2g\, ,
$$
which is a substantial improvement. The main reason for the improvement is the 
use of the socle filtration as opposed to the Jordan-H\"older filtration. 
\end{remark}

\section{Lower Bounds}\label{s:lower}

The issue of finding lower bounds in questions on essential dimension is more subtle.
We fix a rank $r$ and denote by $\bun_{X, \bD}^{r,\xi,ss}$ the semistable locus
of the moduli stack of vector bundles of rank 
$r$, determinant $\xi$ and parabolic structure along $\bD$. We would like to
find a lower bound on its essential dimension.

Suppose that $p^l$ divides $l(\bD)$ where $p$ is a prime. (Recall the 
definition of $l(\bD)$ from Theorem \ref{t:stable}.)
Let $x$ be a $k$-point of $X$; it exists by our ongoing hypothesis on $X$.
Construct a parabolic point
$\bx\,=\,(x,\{p^l,r-p^l\},\{\alpha_1,\alpha_2\})$, where $\alpha_i$ 
are chosen sufficiently small so that
if a parabolic vector bundle is semistable for the datum $\bD\cup\{\bx \}$ then 
the underlying vector bundle is semistable.
It is easy to see one can do this using the definition of parabolic slope.

\begin{theorem}
\label{t:lower}
We have 
$$
\ed(\bun_{X,\bD}^{r,\xi,ss}) \,\ge\,   (r^2-1)(g-1) + p^l + 
 \sum_{y\in\bD} \dim \flag_{y}(\bD)
$$
and
$$
\ed(\bun_{X,\bD}^{r,d,ss}) \,\ge\,   (r^2-1)(g-1)  +  p^l - 1 + g +
 \sum_{y\in\bD} \dim \flag_{y}(\bD)\, .
$$
\end{theorem}

\begin{proof}
Let $\bD'\,=\, \{\bx\} \cup \bD$ be the datum constructed above.
Now the greatest common divisor of the multiplicities is $p^l$. Hence by
Theorem \ref{t:stable},
as we are in the prime power case and we calculate the essential dimension of the stack
$\bun_{X,\bD'}^{r,\xi,ss}$. 
(Note that Theorem \ref{t:stable} does not require
that $X$ have three $k$-rational points that are not parabolic.) So we have
$$
\ed(\bun^{r,\xi,s}_{X,\bD'}) \,=\, p^l + (r^2-1)(g-1) + \ \flag_\bx(\bD')
+ \sum_{y\in\bD'} \dim \flag_{y}(\bD)\, .
$$
The result now follows from the fibration theorem, 
\cite[Theorem 3.2]{brosnan:09}, applied to the representable fibration
$$
\bun^{r,\xi,s}_{X,\bD'} \,\longrightarrow\, \bun^{r,\xi,s}_{X,\bD}.
$$
The non-fixed determinant case is analogous via Corollary \ref{c:stable}.
\end{proof}


\begin{thebibliography}{CTKM07}

\bibitem[BBN01]{balaji:01}
V.~Balaji, I.~Biswas and D.~S. Nagaraj:
Principal bundles over projective manifolds with parabolic structure
over a divisor, {\it Tohoku Math. Jour.} \textbf{53} (2001), 337--367.

\bibitem[BD]{bd:09}
I.~Biswas and A.~Dey:  Brauer group of a moduli space of
parabolic vector bundles over a curve, {\it Jour. K-Theory} \textbf{8} 
(2011), 437--449.

\bibitem[BF03]{berhuy:03} G. Berhuy and G. Favi: Essential 
dimension: a functorial point of view (after {A}. {M}erkurjev), {\it Doc. 
Math.} \textbf{8} (2003), 279--330.

\bibitem[BFRV]{brosnan:09}
P.~Brosnan, Z.~Reichstein and A.~Vistoli: Essential dimension of moduli of 
curves and other algebraic stacks (with an appendix by N. Fakhruddin), 
{\it Jour. Eur. Math. Soc.} \textbf{13} (2011), 1079--1112.

\bibitem[Bis97]{biswas:97} I.~Biswas: Parabolic bundles as orbifold bundles,
{\it Duke Math. Jour.} \textbf{88} (1997), 305--325.

\bibitem[Bor07]{borne:07} N.~Borne, 
Fibr\'es paraboliques et champ des racines,
 {\it Int. Math. Res. Not. IMRN}, \textbf{16}
  2007,  1073--7928
  

\bibitem[BR97]{buhler:97} J.~Buhler and Z.~Reichstein: On the essential 
dimension of a finite group, {\it Compositio Math.} {\bf 106} (1997), 159--179.

\bibitem[CKM07]{colliot:07}
J.-L. Colliot~Thelene, N.~A. Karpenko and A.~S. Merkurev:
Rational surfaces and the canonical dimension of the group {${\rm
PGL}_6$} {\it Algebra i Analiz} {\bf 19} (2007), 159--178.

\bibitem[DL09]{dl:09} A.~Dhillon and N.~Lemire: Upper bounds for the essential 
dimension of the moduli stack of {${\rm SL}_n$}-bundles over a curve,
{\it Transform. Gr.} {\bf 14} (2009), 747--770.

\bibitem[Har77]{hartshorne:77} R.~Hartshorne: {\it Algebraic geometry},
Graduate Texts in Mathematics, No. 52, Springer-Verlag, New York, 1977.

\bibitem[HL97]{huybrechts:97}
D. Huybrechts and M. Lehn: {\it The geometry of moduli spaces of sheaves},
Aspects of Mathematics, E31. Friedr. Vieweg \& Sohn, Braunschweig, 1997.

\bibitem[Kar00]{kar:00}
N.~Karpenko: On anisotropy of orthogonal involutions,
{\it J. Ramanujan Math. Soc.} {\bf 15} (2000), 1--22.

\bibitem[KM97]{keel:97}
S.~Keel and S.~Mori: Quotients by groupoids, {\it Ann. of Math.} 
{\bf 145} (1997), 193--213.

\bibitem[K{\"o}c05]{kock:05}
B.~K{\"o}ck: Computing the equivariant {E}uler characteristic of {Z}ariski and
\'etale sheaves on curves, {\it Homology, Homotopy Appl.}, {\bf 7}
(2005), 83--98.

\bibitem[Lie08]{lieblich:05}
M.~Lieblich: Twisted sheaves and the period-index problem,
{\it Compos. Math.} {\bf 144} (2008), 1--31.

\bibitem[LMB00]{laumon:00}
G.~Laumon and L.~Moret-Bailly: {\em Champs Alg\'{e}briques},
Springer-Verlag, 2000.

\bibitem[Mer03]{merk:03}
A.~Merkurjev: Steenrod operations and degree formulas,
{\it Jour. Reine Angew. Math.} {\bf 565} (2003), 13--26.

\bibitem[Mer09]{merk:09}
A.~S. Merkurjev: Essential dimension,
inn {\it Quadratic forms---algebra, arithmetic, and geometry}, 299--325,
{\em Contemp. Math.}, vol. 493, Amer. Math. Soc., Providence, RI, 2009.

\bibitem[MS80]{mehta:80}
V.~B. Mehta and C.~S. Seshadri:
Moduli of vector bundles on curves with parabolic structures,
{\it Math. Ann.} {\bf 248} (1980), 205--239.

\bibitem[Re1]{reichstein:10}
Z.~Reichstein: Essential dimension, Proceedings of the International Congress 
of Mathematicians, Volume II, 162--188, Hindustan Book Agency, New Delhi.

\bibitem[Re2]{reichstein:00}
Z.~Reichstein: On the notion of essential dimension for algebraic groups,
{\it Transform. Gr.} {\bf 5} (2000), 265--304.

\bibitem[T\"o99]{toen:99}
B.~T\"oen: Th\'eor\`emes de {R}iemann-{R}och pour les champs de
  {D}eligne-{M}umford, {\it $K$-Theory} {\bf 18} (1999), 33--76.

\bibitem[Yo95]{yokogawa:95}
K.~Yokogawa: Infinitesimal deformation of parabolic {H}iggs sheaves,
{\it Internat. Jour. Math.} {\bf 6} (1995), 125--148.

\end{thebibliography}
\end{document}